\documentclass{amsart}
\usepackage{mathscinet,amssymb,latexsym,stmaryrd,phonetic}
\usepackage{enumerate,mathrsfs,hyperref,cmll}
\usepackage{xcolor,centernot}

\newtheorem{theorem}{Theorem}[section]
\newtheorem{proposition}[theorem]{Proposition}
\newtheorem{lemma}[theorem]{Lemma}

\newtheorem{question}[theorem]{Question}
\newtheorem{corollary}[theorem]{Corollary}

\theoremstyle{definition}
\newtheorem{definition}[theorem]{Definition}

\newtheorem{claim}[theorem]{Claim}

\theoremstyle{remark}
\newtheorem{remark}[theorem]{Remark}

\numberwithin{equation}{section}
\newcommand{\onethin}{1\text{-}\mathsf{thin}}
\newcommand{\ZFC}{\mathsf{ZFC}}

\newcommand{\fs}{\mathop{\mathrm{FS}}}

\newcommand{\kthin}{k\text{-}\mathsf{thin}}

\newcommand{\restr}[2]{#1\upharpoonright {#2}}
\newenvironment{proofclaim}[1][Proof of the claim]{\textbf{#1.} }{\hfill \rule{0.5em}{0.5em}}

\begin{document}
\title{$\mathbb{Z}$-Ramsey ultrafilters}



\author[Corona-Garc\'{\i}a]{J. A. Corona-Garc\'{\i}a}
\address{Posgrado Conjunto en Ciencias Matem\'aticas UNAM-UMSNH\\Morelia, Michoac\'an\\ M\'exico 58089.}
\curraddr{}
\email{jcorona@matmor.unam.mx}
\thanks{{The first author's research has been supported by Secihti, scholarship 298332, and by PAPIIT grant IN104419. The research of the second author was partially supported  by grants SIP-20253559 and SIP-20260817 from IPN, as well as by grant CBF2023-2024-334 from Secihti. The third named author was partially supported by PAPIIT grants IN104419 and IN116225.}}

\author[Fern\'{a}ndez-Bret\'{o}n]{D. J. Fern\'{a}ndez-Bret\'{o}n}
\address{Escuela Superior de F\'sica y Matem\'aticas\\ Instituto Polit\'ecnico Nacional\\ Av. Instituto Polit\'ecnico Nacional s/n Edificio 9 \\ Col. San Pedro Zacatenco\\ Alcald\'ia Gustavo A. Madero\\ 07738\\ CDMX\\ M\'exico.}
\curraddr{}
\email{dfernandezb@ipn.mx}
\urladdr{https://dfernandezb.web.app}

\author[Ramos-Garc\'{\i}a]{U. A. Ramos-Garc\'{\i}a}
\address{Centro de Ciencias Matem\'aticas\\ Universidad Nacional Aut\'onoma de M\'exico\\ Campus Morelia\\Morelia, Michoac\'an\\ M\'exico 58089.}
\curraddr{}
\email{ariet@matmor.unam.mx}

 
\date{\today}


\begin{abstract}
We study \(\mathbb{Z}\)-Ramsey ultrafilters, ultrafilters containing witnesses to every shift-invariant instance of Ramsey's theorem. We prove that it is consistent that there are no $\mathbb Z$-Ramsey ultrafilters. We also prove that every \((\mathbb{Z},3)\)-Ramsey ultrafilter, as well as every \(\mathbb{Z}\)-Ramsey P-point, is selective. Further, we exhibit a generic extension—using quotient algebras of the form \(\mathcal{P}(\mathbb{Z})/\mathcal{I}\) for certain \(F_{\sigma }\)-ideals—that contains P-points that are not \(\mathbb{Z}\)-Ramsey ultrafilters, thereby addressing open questions raised by Petrenko and Protasov ~\cite{MR3376354,MR3681104}.
\end{abstract}

\maketitle

\section{ Introduction}

The main object of study in this paper are $\mathbb Z$-Ramsey ultrafilters: ultrafilters containing homogeneous sets for every colouring that is invariant with respect to the usual shift action of $\mathbb Z$ on itself. These ultrafilters were first introduced by Petrenko and Protasov~\cite{MR3376354}, in the more general context of a $G$-Ramsey ultrafilter over any $G$-set $X$ (defined as ultrafilters on $X$ containing $c$-homogeneous sets for every colouring $c$ of $[X]^2$ that is invariant under the action of $G$ on $X$). They also introduced $G$-selective ultrafilters (ultrafilters with the property that for each $G$-invariant partition of $X$, either one of the pieces of the partition belongs to the ultrafilter, or the ultrafilter contains a selector for the partition); in a related paper~\cite{MR3681104} (published two years later but probably written concurrently), the same two authors proved that, in general, any $G$-Ramsey ultrafilter is $G$-selective~\cite[Theorem 5.3]{MR3681104} but the reverse implication does not hold (because, e.g., $G$-selective ultrafilters can always be proven to exist in $\mathsf{ZFC}$~\cite[Theorem 5.1]{MR3681104}, whereas the same cannot be said of $G$-Ramsey ultrafilters, as we will discuss in what follows).

So, more concretely, we will say that an ultrafilter $u$ on $\mathbb Z$ is {\bf $(\mathbb Z,k)$-Ramsey} if, for every colouring $c:[\mathbb Z]^k\longrightarrow 2$ that is invariant under the shift action of $\mathbb Z$ (that is, with the property that $c(\{x_1,\ldots,x_k\})=c(\{x_1+z,\ldots,x_k+z\})$ for every $z\in\mathbb Z$), the ultrafilter contains a $c$-homogeneous set. We will simply say that an ultrafilter is $\mathbb Z$-Ramsey when it is $(\mathbb Z,2)$-Ramsey. One could say that, in the same way that selective ultrafilters are related to Ramsey's theorem, $\mathbb Z$-Ramsey ultrafilters are related to the $\mathbb Z$-Ramsey theorem from~\cite{fb-afml}. The paper~\cite{MR3681104} contains a plethora of results regarding $\mathbb Z$-Ramsey ultrafilters (particularly in terms of necessary conditions, as well as a number of properties of ultrafilters that imply that they are {\it not} $\mathbb Z$-Ramsey ultrafilters). Of particular importance are certain theorems concerning selective ultrafilters, such as~\cite[Theorem 2.6]{MR3376354}, stating that every $\mathbb Z$-Ramsey ultrafilter containing a $1$-thin set (to be defined below) must be selective; in a similar vein,~\cite[Corollary 2.7]{MR3681104} establishes that every $\mathbb Z$-Ramsey ultrafilter that is also a Q-point, must necessarily be selective. Notably, it is also the case that $(\mathbb Z,4)$-Ramsey ultrafilters must be selective~\cite[Corollary 2.8]{MR3681104}; in this paper we improve that result to $(\mathbb Z,3)$-ultrafilters (see Theorem~\ref{thm:z3impliesramsey} below). It is not clear yet that this number 3 can be further improved: one of the main questions motivating the work on this paper is~\cite[Question 5.1]{MR3376354} whether every $\mathbb{Z}$-Ramsey ultrafilter must be selective. In this paper we provide a partial answer to this question, given by a theorem that is reminiscent to several of the results mentioned in this paragraph: Theorem~\ref{z-ramsey+ppointimpliesselective} below establishes that every $\mathbb Z$-Ramsey ultrafilter that is also a P-point must be selective.

Another related question, mentioned in~\cite[p. 458]{MR3681104} (see the short paragraph right after~\cite[Corollary 2.7]{MR3681104}) is whether every P-point in $\mathbb{Z}^{*}$ is a $\mathbb{Z}$-Ramsey ultrafilter. We also address this question in the present paper, answering it in the negative by exhibiting a family of forcing notions (generated by certain $F_\sigma$-ideals) with the property that they provide a generic extension containing P-points (moreover, P-points all of whose elements are positive with respect to the relevant ideal) that are not $\mathbb Z$-Ramsey ultrafilters. In the process of developing this proof, we provide a negative Ramsey-theoretic result by showing that there exist certain sets carrying $\mathbb Z$-invariant colourings without ``large'' monochromatic subsets (with ``large'' interpreted as positive with respect to the ideal at hand). In the end, although not directly relevant to the ultrafilter questions, we complement the negative Ramsey-theoretic result just mentioned with a positive one, stating that for each $\mathbb Z$-invariant colouring of $[\mathbb Z]^2$, it is possible to find a ``large'' (in a sense that will be made precise) monochromatic set, see Theorem~\ref{thm:positivemonochromatic} below.

In Section 2, we summarize some preliminary results and definitions that will be used throughout the paper. In Section 3, we proceed to study some ideals, in particular what we call $F_\sigma$ similarity invariant ideals and, for some particular cases of these ideals, we prove that they do not satisfy the corresponding version of Ramsey's theorem; as a corollary of this, we conclude that it is possible to force the existence of P-points that are not $\mathbb Z$-Ramsey ultrafilters. In section 4 we establish that the existence of $\mathbb Z$-Ramsey ultrafilters is independent from $\ZFC$, and, in Section 5, we show that $\mathbb Z$-Ramsey ultrafilters that are also P-points must be selective. Finally, Section 6 contains a Ramsey-theoretic statement about the ideal generated by $1$-thin sets, and states a related open question.

\section{Preliminaries}

Let us recall the basic notions we will use throughout the paper.

\medskip

A non-empty family $\mathcal{I}$ of subsets of a set $X$ is an {\bf ideal} on $X$ if it is closed under taking subsets and finite unions and does not contain the set $X$. In this paper we assume that all ideals on $X$ contain all finite subsets of $X$. We say that an ideal $\mathcal{I}$ on $X$ is {\bf tall} if for each $Y \in [X]^{\omega}$ 
there exists $I \in \mathcal{I}$ such that $I \cap Y$ is infinite. Given an ideal $\mathcal{I}$ on a set $X$, we denote by $\mathcal{I}^{+}$ the family of $\mathcal{I}$-positive sets, that is, subsets of $X$ which are not in $\mathcal{I}$. If $\mathcal{I}$ is an ideal on $X$ and $Y \in \mathcal{I}^{+}$, we denote by $\restr{\mathcal{I}}{Y}$ the ideal $\{I \cap Y \mid  I \in \mathcal{I}\}$ on $Y$. When $X$ is countable, an ideal $\mathcal{I}$ on $X$ can be regarded as an ideal on $\omega$ via any bijection between $X$ and $\omega$, so we will assume our ideals to be ideals on $\omega$.

We consider $\mathcal{P}(\omega)$ equipped with the natural topology induced by identifying each subset of $\omega$ with its characteristic function, where $2^{\omega}$ is given the product topology. We call an ideal $\mathcal{I}$ a Borel (analytic, co-analytic,\dots) ideal on $\omega$ if $\mathcal{I}$ is an ideal on $\omega$ and $\mathcal{I}$ is Borel (analytic, co-analytic,\dots) in this topology.


There is an extremely close and useful connection between $F_{\sigma}$ ideals and lower semicontinuous submeasures. A {\bf submeasure on} a set $X$ is a function $\varphi \colon \mathcal{P}(X) \to [0, \infty]$ satisfying $\varphi(\emptyset)=0$, if $A \subseteq B$ then $\varphi(A) \leqslant \varphi(B)$, and $\varphi( A\cup B) \leqslant \varphi(A) + \varphi(B)$. To avoid trivialities, we also require that $\varphi(F) < \infty$ for all finite subset $F$ of $X$. If $\varphi$ is a submeasure on $\omega$ and satisfies $\varphi(A)=\lim_{n \to \infty}\varphi(A \cap n)$ for all $A \subseteq \omega$ then $\varphi$ is called a {\bf lower semicontinuous submeasure}, abbreviated by lscsm. To each lscsm $\varphi$ on $\omega$ naturally correspond the ideal $\mathsf{Fin}(\varphi)=\{A \subseteq \omega \mid  \varphi(A) < \infty\}$. It is immediate from the definition that $\mathsf{Fin}(\varphi)$ is an $F_{\sigma}$ ideal. The following fundamental theorem of Mazur is key to the study of $F_{\sigma}$-ideals.

\begin{theorem}[\cite{MR1124539}]\label{Thm:Mazur}
   Let $\mathcal{I}$ be an ideal on $\omega$. Then $\mathcal{I}$ is an $F_{\sigma}$ ideal if and only if there is a lscsm $\varphi$ such that $\mathcal{I}=\mathsf{Fin}(\varphi)$.\hfill $\square$
\end{theorem}

A useful characterization of selective ultrafilters is due to Mathias:

\begin{theorem}[\cite{MR491197}]\label{T:Mathias-Selective}
    Let $u$ be an ultrafilter on $\omega$. Then, $u$ is selective if and only if $u \cap \mathcal{I} \ne \emptyset$ for every analytic tall ideal $\mathcal{I}$ on $\omega$.\hfill $\square$
\end{theorem}


Various types of ultrafilters can be added generically by the quotient algebra $\mathcal{P}(\omega)/\mathcal{I}$, where $\mathcal{I}$ is some definable ideal (usually one such that $\mathcal{P}(\omega)/\mathcal{I}$ does not add reals; for example, $F_{\sigma}$-ideals). Indeed, using Theorem \ref{Thm:Mazur}, it is easy to prove the following result, which was first observed in \cite{MR748847}:

\begin{remark}[\cite{MR748847}]\label{Remark:Just-Krawczyk}
    If $\mathcal{I}$ is an $F_{\sigma}$ ideal, then $\mathcal{P}(\omega)/\mathcal{I}$ is $\sigma$-closed; in fact, $\mathcal{I}$ is a $\text{P}^{+}$-ideal.\hfill $\square$
\end{remark}

An ideal $\mathcal{I}$ on $\omega$ is a {\bf $\text{P}^{+}$-ideal} if for every decreasing sequence $\{X_{n} \mid  n < \omega\}$ of $\mathcal{I}$-positive sets there is an $\mathcal{I}$-positive set $X$ such that $X \subseteq^{*} X_{n}$, for all $n < \omega$.

The aforementioned quotient algebra provides a straightforward method for the consistent construction of various types of ultrafilters \cite{MR2861027}. Instead of dealing with the quotient algebra $\mathcal{P}(\omega)/\mathcal{I}$, it is common to use the non-separative equivalent forcing notion $(\mathcal{I}^{+}, \subseteq)$.

Following \cite{MR3692233}, we will be interested in the following $G$-invariant Ramsey type properties of ideals over a group $G$: Let $\mathcal{I}$ be an ideal on a countably infinite group $G$. We say that $\mathcal{I}$ satisfies
\[
  G \overset{\text{left-inv.}}{\xrightarrow{\hspace{1.1cm}}} (\mathcal{I}^{+})^{2}_{2}
\]
if for every colouring $c \colon [G]^{2} \longrightarrow 2$ that is invariant under the left multiplication action of $G$ (that is, with the property that $c(\{x_{1},x_{2}\})=c(\{gx_{1},gx_{2}\})$ for every $g \in G$), there is an $\mathcal{I}$-positive set $X$ homogeneous with respect to $c$. We say that $\mathcal{I}$ satisfies
\[
  \mathcal{I}^{+} \overset{\text{left-inv.}}{\xrightarrow{\hspace{1.1cm}}} (\mathcal{I}^{+})^{2}_{2}
\]
if for every $\mathcal{I}$-positive set $X$ and every left-invariant colouring $c \colon [G]^{2} \longrightarrow 2$ there is an $\mathcal{I}$-positive subset $Y\subseteq X$ homogeneous with respect to $c$. Thus, we shall call an ideal $\mathcal{I}$ {\bf left-invariant Ramsey$(G)$}, if 
$G \overset{G\text{-inv.}}{\xrightarrow{\hspace{1.1cm}}} (\mathcal{I}^{+})^{2}_{2}$, and we shall say that $\mathcal{I}$ is {\bf left-invariant Ramsey} if $\mathcal{I}^{+} \overset{G\text{-inv.}}{\xrightarrow{\hspace{1.1cm}}} (\mathcal{I}^{+})^{2}_{2}$. Let $X$ be an $\mathcal{I}$-positive set. A partition $\{F_{n} \mid  n < \omega\}$ of $X$ into finite sets is called {\bf thin} if, for every $n > 0$ and $F=\bigcup_{k<n}F_{k}$, the following conditions are satisfied: $D(F_{n},F_{n}) \cap D(F, F) = \emptyset$, $D(F_{n},F_{n}) \cap D(F_{n},F) = \emptyset$, and $D(F_{n}, F) \cap D(F, F) = \emptyset$, where $D$ is defined as follows.

\begin{definition}
    Given a group $G$ and subsets $A,B\subseteq G$, we denote $D(A,B)=\{y^{-1}x \mid x \in A, \, y \in B \text{ and } x \ne y\}$. When $A=B$, we simply write $D(A)$.
\end{definition}

The following definition delineates the main object of study in this paper.

\begin{definition}
Given an $n\in\mathbb N$, an ultrafilter $u\in\beta\mathbb Z$ is called {\bf $(\mathbb Z,n)$-Ramsey} if, for every colouring $c:[\mathbb Z]^n\longrightarrow 2$ that is $\mathbb Z$-invariant, there exists an $A\in u$ such that $[A]^n$ is $c$-monochromatic. In case $n=2$, we will simply say that $u$ is a {\bf $\mathbb Z$-Ramsey ultrafilter}.
\end{definition}

Note that every ultrafilter is $(\mathbb Z,1)$-Ramsey. By the usual arguments involving the dropping of one coordinate of a tuple, every $(\mathbb Z,n+1)$-Ramsey ultrafilter is seen to be a $(\mathbb Z,n)$-Ramsey ultrafilter. Petrenko and Protasov in~\cite{MR3681104} proved that every $(\mathbb Z,4)$-Ramsey ultrafilter must be selective.


\section{A generic P-point not \texorpdfstring{$\mathbb{Z}$}{Lg}-Ramsey}

In this section, we develop the necessary tools to negatively answer the question mentioned in~\cite[p. 458]{MR3681104}, showing that for certain $F_{\sigma}$-ideals $\mathcal{I}$ on $\mathbb{Z}$, the quotient algebra $\mathcal{P}(\mathbb{Z})/\mathcal{I}$ generically adds an ultrafilter $u$ which is a P-point but not $\mathbb{Z}$-Ramsey.

Recall that a subset $A$ of a countably infinite group $G$ is $k\textbf{-thin}$ ($k \geq 1$) if 
\[
  |gA \cap A|\leq k
\]
for each $g \in G\setminus \{e_{G}\}$ (see \cite{MR3681104, MR1834643}). In this context, we define the natural ideals generated by such sets as follows:
\[
\mathcal{I}_{\kthin}=\left\{I \subset G \mid  
\left(\exists \{A_{i}\mid  i < n\} \subset \kthin\right)
\left( I \subseteq \bigcup_{i<n}A_{i}\right)\right\},
\]
where $\kthin = \{A \subset G \mid  A \text{ is a $k$-thin set}\}$. The family $\kthin$ is hereditary; that is, for each $A \in \kthin$ and each $B \subseteq A$, it follows that $B \in \kthin$. Also, note that
\[
\mathcal{I}_{\kthin} \subseteq \mathcal{I}_{\text{$(k+1)$-$\mathsf{thin}$}}
\]
for every $k \geq 1$.

\begin{lemma} Let $G$ be a countably infinite group. Then $\mathcal{I}_{\kthin}$ is a tall $F_{\sigma}$ ideal. 
\end{lemma}
\begin{proof}
To see that $\mathcal{I}_{\kthin}$ is tall, it suffices to prove that $\mathcal I_{1\text{-}\mathsf{thin}}$ is tall. Let $X\subseteq G$ be an infinite set. Recursively choose, for each $n<\omega$, 
\begin{equation*}
    y_{n}\in X\setminus\left(\{y_i y_j^{-1} y_k \mid   i,j,k< n\}\cup\{y_i\mid  i< n\}\right),
\end{equation*}
and let $Y=\{y_n\mid  n<\omega\}$. By construction, $Y\subseteq X$ is an infinite set; we claim that this set is $1$-thin. To see this, let $g\in G\setminus\{e_{G}\}$ be arbitrary and assume, to get a contradiction, that $|gY\cap Y|\geq 2$. This means there are two distinct $i,j$ such that $y_i,y_j\in gY$, so there are $l,k$ such that $y_i=gy_k$ and $y_j=gy_l$. Note that this necessarily implies that $i,j,k,l$ are pairwise distinct, and $g=y_i y_k^{-1}=y_j y_l^{-1}$. By taking inverses on both sides of the last equation if necessary, we may assume that $i=\max\{i,j,k,l\}$, and from here we may deduce $y_i=y_j y_l^{-1} y_k$, contrary to the recursive definition of the $y_n$.

On the other hand, as the function $\cup \colon \mathcal{P}(G) \to \mathcal{P}(G)$ given by 
\[
  \cup(A, B) = A \cup B
\]
is continuous and the family $\kthin$ is hereditary, to check that $\mathcal{I}_{\kthin}$ is $F_{\sigma}$, it suffices to show that  $\kthin$ is closed in $\mathcal{P}(G)$. Take any $A \in \mathcal{P}(G)\setminus \kthin$. Then there exists a $g\in G \setminus \{e_{G}\}$ such that $\lvert gA\cap A\rvert>k$. Pick $a_0,a_1,\dots,a_k\in A$ such that  $ga_0,ga_1,\dots,ga_k\in A$. Observe that $U=\{X\in \mathcal{P}(G) \mid   \{a_i \mid   i\leq k\}\cup\{ga_i \mid   i\leq k\}\subseteq X\}$ is open, $U\cap \kthin=\emptyset$, and $A\in U$. Therefore, $\kthin$ is a closed set.
\end{proof}

Since the family $\text{$k$-$\mathsf{thin}$}$ is hereditary and closed in $\mathcal{P}(G)$, it follows from Theorem \ref{Thm:Mazur} that we can provide a witness lscsm, $\varphi_{\text{$k$-$\mathsf{thin}$}}$, showing that the ideal $\mathcal{I}_{\text{$k$-$\mathsf{thin}$}}$ is an $F_{\sigma}$ ideal. To this end, we define  $F_{0}=\text{$k$-$\mathsf{thin}$}$ and $F_{n}=\{A \cup B \mid  A, B \in F_{n-1}\}$ for 
$n \geq 1$. Thus, $F_{n} \subseteq F_{n+1}$ ($n<\omega$) and $\mathcal{I}_{\text{$k$-$\mathsf{thin}$}} = \bigcup_{n < \omega} F_{n}$. For any $E \in [G]^{<\omega}$, we set  
\[
\varphi_{\text{$k$-$\mathsf{thin}$}}(E)= \min \{n +1 \mid  E \in F_{n}\},
\]
and define $\varphi_{\text{$k$-$\mathsf{thin}$}}(A) :=\lim_{n \to \infty} \varphi_{\text{$k$-$\mathsf{thin}$}}(A \cap \{g_{i} \mid  i < n\})$ for any $A \in \mathcal{P}(G)$, where $\{g_{n} \mid  n < \omega\}$ is a fixed enumeration of $G$. Moreover, since $\text{$k$-$\mathsf{thin}$}$ is left-invariant (meaning $g A \in \text{$k$-$\mathsf{thin}$}$ for all $A \in \text{$k$-$\mathsf{thin}$}$ and $g \in G$), it follows that both $\varphi_{\text{$k$-$\mathsf{thin}$}}$ and the ideal $\mathcal{I}_{\text{$k$-$\mathsf{thin}$}}$ are left-invariant. 

\medskip

The $1$-thin sets play a crucial role in determining when a $\mathbb{Z}$-Ramsey ultrafilter is selective, as shown by the following result.

\begin{theorem}[\cite{MR3376354}, Theorem 2.6]\label{Zramsey+thinset-implica-ramsey}
Let $u$ be a $\mathbb Z$-Ramsey ultrafilter containing a $1$-thin set. Then, $u$ must be selective.\hfill $\square$
\end{theorem}

On the other hand, Mathias's result (Theorem \ref{T:Mathias-Selective}) tells us that containing a 1-thin set is, in fact, a necessary condition for a $\mathbb{Z}$-Ramsey ultrafilter to be selective. Thus, a $\mathbb{Z}$-Ramsey ultrafilter $u$ is selective if and only if $u \cap \mathcal{I}_{\onethin} \ne \emptyset$.

\begin{definition}\label{Def-SimInvIdeal}
    Let $G$ be a countably infinite group. An $F_{\sigma}$ ideal $\mathcal{I}$ on $G$ (witnessed by a lscsm $\varphi$) is called a {\bf similarity-invariant ideal} if the following conditions hold:
    \begin{enumerate}[(i)]
        \item $\varphi$ is left-invariant; that is, $\varphi(A)=\varphi(g A)$ for all $A \subseteq G$ and all $g \in G$.
        \item For every $F \in [G]^{<\omega}$ and every $n < \omega$, there exists $E \in [G]^{<\omega}$ such that $\varphi(E) \geq n$, $E \cap F = \emptyset$, $D(E,E) \cap D(F,F) = \emptyset$, $D(E,E) \cap D(E,F) = \emptyset$ and $D(E,F) \cap D(F,F) = \emptyset$.     
    \end{enumerate}
\end{definition}

Examples of similarity-invariant ideals relevant to our purposes are the van der Waerden ideal $\mathcal{W}$ and the $k$-thin ideal $\mathcal{I}_{\kthin}$ over the group of integers $\mathbb{Z}$. Recall that a set $A \subseteq \mathbb{Z}$ is called an \textbf{AP-set} if it contains arbitrarily long arithmetic progressions. The \textbf{van der Waerden ideal} is defined by
\[
  \mathcal{W} = \{A \subseteq \mathbb{Z} \mid A \text{ is not an AP-set}\}.
\]
It is well known that $\mathcal{W}$ is a tall $F_{\sigma}$-ideal on $\mathbb{Z}$. Note that $\mathcal{W}$ is an invariant ideal and, since each $A \in \kthin$ contains no arithmetic progressions of length $k+2$, it follows that $\mathcal{I}_{\kthin} \subseteq \mathcal{W}$, and hence $\mathcal{W}^{+} \subseteq \mathcal{I}^{+}_{\kthin}$.

\begin{lemma}\label{L: similarity-invariant}
    If $G$ is the group of integers $\mathbb{Z}$, then both $\mathcal{W}$ and $\mathcal{I}_{\kthin}$ are similarity-invariant ideals.
\end{lemma}

\begin{proof}
    We only need to show that $\mathcal{W}$ satisfies clause (ii) from Definition \ref{Def-SimInvIdeal}.
    Let $F \in [\mathbb{Z}]^{<\omega}$ and $n < \omega$. Let $\varphi_{\mathcal{W}}$ be a lscsm witness that $\mathcal{W}$ is an $F_{\sigma}$-ideal. First, observe that $a + d \mathbb{N} = \{a + dx \mid x \in \mathbb{N}\} \in \mathcal{W}^{+}$ for all $a \in \mathcal{\mathbb{Z}}$ and $d \in \mathbb{N}$. Take $a_{F} > \max\{|x| \mid x \in F\}$ and $d_{F} > \max\{a_{F} - x \mid x \in F \}$. Since $a_{F} + d_{F} \mathbb{N} \in \mathcal{W}^{+}$, we can find a finite subset $E$ of $a_{F} + d_{F} \mathbb{N}$ with $\varphi_{\mathcal{W}}(E) \geq n$. Clearly, $E \cap F = \emptyset$. Since $a_{F}-x= (a_{F} -y) + (y -x) > y -x$ if $x, y \in F$ and $x < y$ and, $D(E, E) \subseteq \pm d_{F} \mathbb{N}$, it follows that $D(E,E) \cap D(F,F) = \emptyset$. To see that $D(E,E) \cap D(E,F) = \emptyset$, assume for contradiction that there exist $x \in F$  and 
    $y \in \mathbb{N}$ with  $a_{F} + d_{F}y \in E$ such that $(a_{F} + d_{F}y) - x =  d_{F}z$ for some $z \in \mathbb{N}$. Then $a_{F} - x = d_{F} (z-y)$; this implies $d_{F}$ is a divisor of $a_{F}-x$.  In particular, $d_{F} \leq a_{F}-x$, which contradicts $d_{F} > \max\{a_{F} - x \mid x \in F\}$. Finally, as $z -x < (z-x) + (a_{F} -z) + d_{F}y = (a_{F} + d_{F}y) -x$ if $x, z \in F$ with $x < z$ and $y \in \mathbb{N}$, we conclude that $D(E,F) \cap D(F,F) = \emptyset$.
\end{proof}

\begin{lemma}\label{L: invariant-partition}
    Let $\mathcal{I}$ be an $F_{\sigma}$ similarity-invariant ideal on a countably infinite group $G$. There exists an $\mathcal{I}$-positive set $X$, a thin partition $\{F_{n} \mid  n < \omega\}$ of $X$ consisting of finite sets and a left-invariant colouring $c \colon [G]^{2} \longrightarrow 2$ such that every $c$-homogeneous subset $Y$ of $X$ is either finite or a partial selector of $\{F_{n} \mid  n < \omega\}$ (that is, $|Y \cap F_{n}| \leq 1$ for all $n < \omega$). 
\end{lemma}




\begin{proof}
    Let $\varphi$ be a lscsm witnessing the fact that $\mathcal{I}$ is a similarity-invariant ideal. Using clause (ii) from Definition \ref{Def-SimInvIdeal}, we can recursively construct a family $\{F_{n} \mid  n < \omega \}$ of finite subsets of $G$ such that:
    \begin{enumerate}
        \item $\varphi(F_{n}) \geq n$ for $n <   \omega$.
        \item $F_{n} \cap F_{m} = \emptyset$ for $n<m<\omega$.
        \item If $n > 0$ and $F=\bigcup_{k<n}F_{k}$, then $D(F_{n},F_{n}) \cap D(F, F) = \emptyset$, $D(F_{n},F_{n}) \cap D(F_{n},F) = \emptyset$, and $D(F_{n}, F) \cap D(F, F) = \emptyset$.
    \end{enumerate}
    
    We set $X = \bigcup_{n < \omega} F_{n}$. Since $\varphi(F_{n}) \to \infty$, it follows that $X \in \mathcal{I}^{+}$. Define the colouring $c \colon [G]^{2} \longrightarrow 2$ by $c(\{x,y\})=0$ if there exist $n < \omega$ and $z,w \in F_{n}$ such that $z^{-1}w = x^{-1}y$. Clearly, $c$ is left-invariant. By clause (3), any $c$-homogeneous subset $Y$ of $X$ is either homogeneous in colour 0 and finite (contained within one single $F_{n}$), or infinite, in which case it is homogeneous in colour 1 and $Y$ is a partial selector for the family $\{F_{n} \mid  n < \omega \}$.
\end{proof}

We next demonstrate that when $G$ is the group of integers $\mathbb{Z}$, both $\mathcal{W}$ and $\mathcal{I}_{\text{$k$-$\mathsf{thin}$}}$ fail to be $\mathbb{Z}$-invariant Ramsey.

\begin{theorem}\label{T:Not Z-invariant Ramsey}
    There exists a $\mathcal{W}$-positive set (resp., $\mathcal{I}_{\kthin}$-positive set) $X \subseteq \mathbb{Z}$ and a $\mathbb{Z}$-invariant colouring $c \colon [\mathbb{Z}]^{2} \longrightarrow 2$ such that every $c$-homogeneous set $Y \subseteq X$ belongs to $\mathcal{W}$ (resp., $\mathcal{I}_{\kthin}$); that is, 
    \[
      \mathcal{W}^{+} \overset{\text{left-inv.}}{ \centernot{\xrightarrow{\hspace{1.1cm}}}} \left(\mathcal{W}^{+}\right)^{2}_{2}
      \left(\text{resp., }\mathcal{I}_{\kthin}^{+} \overset{\text{left-inv.}}{ \centernot{\xrightarrow{\hspace{1.1cm}}}} \left(\mathcal{I}_{\kthin}^{+}\right)^{2}_{2}\right).
    \]
\end{theorem}

\begin{proof}
    According to Lemmas \ref{L: similarity-invariant} and \ref{L: invariant-partition}, there exists a $\mathcal{W}$-positive set $X$, a thin partition $\{F_{n} \mid  n < \omega\}$ of $X$ consisting of finite sets and a left-invariant colouring $c \colon [\mathbb{Z}]^{2} \longrightarrow 2$ such that every $c$-homogeneous subset $Y$ of $X$ is either finite or a partial selector of $\{F_{n} \mid  n < \omega\}$. 
    \begin{claim}
        If $Y$ is a partial selector of $\{F_{n} \mid  n < \omega\}$, then $Y \in \onethin$.
    \end{claim}

    \begin{proofclaim}
         Let $x \in \mathbb{Z} \setminus \{0\}$ be arbitrary and assume, for the sake of contradiction, that $|(x + Y) \cap Y| \geq 2$. Then, there are two distinct elements $y_{0},y_{1} \in (x + Y) \cap Y$. Thus, there exist $y_{2},y_{3} \in Y$ such that $y_{0}=x +y_{2}$ and $y_{1}= x + y_{3}$. Note that this implies that $y_{0},y_{1},y_{2},y_{3}$ are pairwise distinct, and $x=y_{0}-y_{2}=y_{1}-y_{3}$. By taking negatives on both sides of the last equation if necessary, we may assume that $y_{0} \in F_{n}$, where $n=\max\{k \mid |F_{k} \cap \{y_{i} \mid i < 4\}| = 1 \}$. Thus, since $Y$ is a partial selector of $\{F_{n} \mid  n < \omega\}$, we may deduce that $(F_{n} \cap F) \cap (F \cap F) \ne \emptyset$, where $F=\bigcup_{k<n}F_{k}$, contradicting the fact that the partition $\{F_{n} \mid  n < \omega\}$ is thin.
    \end{proofclaim}
    
    As $\onethin \subseteq \mathcal{I}_{\kthin} \subseteq \mathcal{W}$ for $k \geq 1$, the result follows. 
\end{proof}




\begin{corollary}\label{T:P-PointNoZ-Rasey}
Let $u$ be the 
$\mathcal{P}(\mathbb{Z})/\mathcal{I}_{\kthin}$-generic ultrafilter; then $u$ is a P-point that is not $\mathbb{Z}$-Ramsey.
\end{corollary}

\begin{proof}
    By Observation 2.4 from \cite{MR2861027}, it follows that $u$ is a P-point and, since the quotient algebra $\mathcal{P}(\mathbb{Z})/\mathcal{I}_{\kthin}$ does not add new reals, Theorem \ref{T:Not Z-invariant Ramsey} implies that there exists a condition $X \in \mathcal{I}^{+}_{\kthin}$ such that $X$ forces that $\dot{u}$ is not $\mathbb{Z}$-Ramsey. 
\end{proof}

\section{A \texorpdfstring{$\mathbb{Z}$}{Lg}-Ramsey ultrafilter implies a P-point}

The question whether the existence of a $\mathbb Z$-Ramsey ultrafilter is independent from $\ZFC$ is a very natural one. The arguments from~\cite[p. 459]{MR3681104} show that the existence of a $B$-Ramsey ultrafilter, where $B$ is a Boolean group (and the notion of $B$-Ramsey is defined analogously), imply the existence of a P-point; however, no such arguments have been made for $\mathbb Z$-Ramsey ultrafilters. In this section, we proceed to prove that, if there exists a $\mathbb Z$-Ramsey ultrafilter, then there is a P-point. In particular, it is consistent with $\ZFC$ that there are no $\mathbb Z$-Ramsey ultrafilters.

From now on, we will assume that our ultrafilters are over $\mathbb N$, the set of positive integers. We do not lose generality by doing so, since every ultrafilter over $\mathbb Z$ must contain either $\mathbb N$ or $-\mathbb N$, and the mapping $x\longmapsto -x$ is a bijection between these two sets that preserves addition. We now recall two key definitions.

\begin{definition}[Petrenko and Protasov,~\cite{MR3681104}, p. 456]
    If $u$ is an ultrafilter, then $D(u)$ is defined to be the filter generated by all sets of the form $D(A)$, with $A\in u$.
\end{definition}

In general, $D(u)$ need not be an ultrafilter; the following theorem will be crucial for us.

\begin{theorem}[Petrenko and Protasov,~\cite{MR3681104}, Thm. 2.2 (ii)]
    Let $u$ be a free ultrafilter on $\mathbb Z$ such that $\mathbb Z^{+} \in u$. Then, $D(u)$ is an ultrafilter if and only if $u$ is $\mathbb Z$-Ramsey.\hfill $\square$
\end{theorem}

The other main tool we will use are the functions $\lambda,\mu:\mathbb N\longrightarrow\omega$. These are defined in the following way: given an $n\in\mathbb N$, and after expanding it in binary notation, $\lambda(n)$ is the position of the least non-zero digit and $\mu(n)$ is the position of the last one. Another equivalent definition is that $\lambda(n)$ is the greatest $k$ such that $2^k\mid  n$, and $\mu(n)=\lfloor\log_2(n)\rfloor$ (and of course, the main intuition is that, if we are going to think of $n$ as a finite set, then $\lambda(n)$ and $\mu(n)$ are simply the minimum and maximum elements, respectively). The main properties of the functions $\lambda$, $\mu$ that we will use here (all of which can be proven easily and elementarily) are the following.

\begin{proposition}\label{prop:propertieslambdamu}\hfill
    \begin{enumerate}
        \item The function $\mu$ is finite-to-one (in fact, for a given $k$ there are exactly $2^k$ distinct values of $n$ such that $\mu(n)=k$),
        \item if $\mu(n)<\mu(m)$ then $\mu(n+m)=\mu(m)$,
        \item if $\mu(n)=\mu(m)$ then $\mu(m+n)=\mu(n)+1$,
        \item if $\lambda(n)<\lambda(m)$ then $\lambda(n+m)=\lambda(n)$,
        \item if $\lambda(n)=\lambda(m)=k$ then $\lambda(n+m)>k$.\hfill $\square$
    \end{enumerate}
\end{proposition}


\begin{lemma}\label{lem:mulambdainjective}
    Let $u$ be a $\mathbb Z$-Ramsey ultrafilter. Then, there is a set $X\in u$ such that, for all $a,b,c\in X$ with $a<b<c$, 
    \begin{enumerate}
        \item $\mu(c-b)\neq\mu(b-a)$, and
        \item $\lambda(c-b)\neq\lambda(b-a)$.
    \end{enumerate}
\end{lemma}

\begin{proof}
    Consider the $\mathbb Z$-invariant colouring $c:[\mathbb N]^2\longrightarrow 2^3$ given by defining $c(x,y)=(i,j,k)$ where $i\equiv\mu(y-x)\mod 2$, $j\equiv\lambda(y-x)\mod 2$, and $k$ is the $(\lambda(y-x)+1)$-st binary digit of $y-x$ (equivalently, if $y-x=2^{\lambda(y-x)}z$, where $z$ is odd, then $2k+1\equiv z\mod 4$). By hypothesis, $u$ is a $\mathbb Z$-Ramsey ultrafilter, so there exists a $c$-monochromatic set $X\in u$; we claim that $X$ is as sought. To see this, take any $a,b,c\in X$ with $a<b<c$. Then, if we had $\mu(c-b)=\mu(b-a)=m$, we would have $\mu(c-a)=\mu((c-b)+(b-a))=m+1$, and so the parity of $\mu(c-a)$ is opposite to that of $\mu(c-b)$, contradicting $c$-monochromaticity. Similarly, if $\lambda(c-b)=\lambda(b-a)=l$, then we could write $b-a=2^lA$ and $c-b=2^lB$, with $A\equiv B\mod 4$ and both $A,B$ odd. This means $A+B\equiv 2\mod 4$, and therefore, from
    \begin{equation*}
        c-a=(c-b)+(b-a)=2^l(A+B),
    \end{equation*}
    we can deduce that $\lambda(c-a)=l+1$, again contradicting $c$-monochromaticity of $l$.
\end{proof}

\begin{lemma}\label{lem:fin-to-one}
    Let $u$ be any ultrafilter, and let $X\in U$ be such that, for all $a,b,c\in X$, $\lambda(c-b)\neq\lambda(b-a)$ and $\mu(c-b)\neq\mu(b-a)$. If $\{x_n \mid n<\omega\}$ is the increasing enumeration of $X$, then:
    \begin{enumerate}
        \item for each $a\in D(X)$, there is an $n<\omega$ such that $\lambda(a)=\lambda(x_{n+1}-x_n)$,
        \item for each $a\in D(X)$, there is an $n<\omega$ such that $\mu(a)=\mu(x_{n+1}-x_n)$, and 
        \item the function $n\longmapsto\lambda(x_{n+1}-x_n)$ is finite-to-one.
    \end{enumerate}
\end{lemma}

\begin{proof}
    We prove points (1) and (2) at once, so let $\kappa$ denote either $\mu$ or $\lambda$, accordingly. Given any $a\in D(X)$, there are $n<m$ such that $a=x_m-x_n$. The proof is by induction on $m-n$; there is nothing to prove if $m-n=1$. Now suppose $a=x_{m+1}-x_n$ and the result holds for $x_m-x_n$. By hypothesis $\kappa(x_{m+1}-x_m)\neq\kappa(x_m-x_n)$, so $\kappa(x_{m+1}-x_n)=\kappa(x_{m+1}-x_m+x_m-x_n)$ equals either $\kappa(x_{m+1}-x_m)$, or $\kappa(x_m-x_n)$ (the maximum of those two if $\kappa=\mu$; the minimum if $\kappa=\lambda$). In the former case, we are done; in the latter, by induction hypothesis there is $k<\omega$ such that $\kappa(x_{m+1}-x_n)=\kappa(x_m-x_n)=\kappa(x_{k+1}-x_k)$, and we are done.

    Now to prove point (3), we will prove by induction on $k$ that there are only finitely many $n<\omega$ such that $\lambda(x_{n+1}-x_n)=k$. So suppose the result is true for all $k'<k$; then, there exists an $n'$ such that whenever $n>n'$, we have $\lambda(x_{n+1}-x_n)\geq k$. Assume there are two distinct $n,m$ with $n'<n<m$ such that $\lambda(x_{m+1}-x_m)=\lambda(x_{n+1}-x_n)=k$; suppose further that $n,m$ were chosen so that $m-n$ is minimal. Then, $\lambda(x_{t+1}-x_t)>k$ for all $n<t<m$. We thus have
    \begin{eqnarray*}
        \lambda(x_m-x_n) & = & \lambda(x_m-x_{m-1}+x_{m-1}-x_{m-2}+\cdots+x_{n+1}-x_n) \\
        & = & \lambda(x_{n+1}-x_n)=k,
    \end{eqnarray*}
    so that $k=\lambda(x_{m+1}-x_m)=\lambda(x_m-x_n)$, a contradiction. Therefore, there is at most one value of $n>n'$ with $\lambda(x_{n+1}-x_n)=k$; overall, there are only finitely many $t$ with $k=\lambda(x_{t+1}-x_t)$.
\end{proof}

\begin{theorem}
If $u$ is a $\mathbb Z$-Ramsey ultrafilter, then $\mu(D(u))$ is a P-point.
\end{theorem}

\begin{proof}
Let $f:\omega\longrightarrow\omega$ be an arbitrary function. We will show that $f$ is either constant, or finite-to-one, on a set in $\mu(D(u))$. Consider the partition $\omega=A_0\cup A_1$, where
\begin{eqnarray*}
     A_0 & = & \{n\in\mathbb N \mid f(\mu(n))<\lambda(n)\}, \\
     A_1 & = & \{n\in\mathbb N \mid \lambda(n)\leq f(\mu(n))\},
\end{eqnarray*}
let $i\in 2$ be such that $A_i\in D(u)$, and let $X\in u$ be such that $D(X)\subseteq A_i$; by Lemma~\ref{lem:mulambdainjective}, we may assume that whenever $a,b,c\in X$ with $a<b<c$, we have $\mu(c-b)\neq\mu(b-a)$ and $\lambda(c-b)\neq\mu(b-a)$. Let $X=\{x_n \mid n<\omega\}$ be the increasing enumeration of $X$, and note that there are two cases

\begin{description}

\item[Case $A_1\in u$] In this case, we will show that $f\upharpoonright \mu[D(X)]$ is finite-to-one. By Lemma~\ref{lem:fin-to-one} (2) we have $\mu[D(X)]=\{\mu(x_{n+1}-x_n) \mid n<\omega\}$. Now, given a $k<\omega$, by Lemma~\ref{lem:fin-to-one} (3) there is an $N$ such that for $n\geq N$ we have $\lambda(x_{n+1}-x_n)>k$; for each such $n$ we may thus conclude that $f(\mu(x_{n+1}-x_n))\geq\lambda(x_{n+1}-x_n)>k$ since $D(X)\subseteq A_1$. That is, $f^{-1}[\{k\}]\cap D(X)$ is finite (of cardinality at most $N$); since the previous reasoning holds for all $k$, we conclude $f$ is finite-to-one in $\mu[D(X)]$.

\item[Case $A_0\in u$] In this case we will find a $Z\in u$ such that $f$ is constant in $\mu[D(Z)]$. By Lemma~\ref{lem:fin-to-one} (3) and possibly dropping finitely many elements of $X$, we may assume without loss of generality that $k_0=\lambda(x_1-x_0)<\lambda(x_{n+1}-x_n)$ for all $n<\omega$. From here, one may easily prove by induction on $n<\omega$ that $\lambda(x_n-x_0)=k_0$ for all $n\geq 2$. Now let $Y=\{x_n \mid n\geq 2\}$, and consider the partition $D(Y)=B_0\cup B_1$ given by
\begin{eqnarray*}
    B_0=\big\{a\in D(Y) & \mid & \text{for some } m>n\geq 2,\ a=x_m-x_n \\
    & & \text{ and }\mu(x_m-x_n)>\mu(x_n-x_0)\big\}, \\
    B_1=\big\{a\in D(Y) & \mid & \text{for all } m>n\geq 2\text{ such that }a=x_m-x_n, \\
    & & \mu(x_m-x_n)<\mu(x_n-x_0)\big\}. 
\end{eqnarray*}
There is $j\in 2$ such that $B_j\in D(u)$; we will now argue that $j=1$ is impossible. Suppose $B_1\in D(u)$, and let $W\subseteq Y$ be such that $D(W)\subseteq B_1$. Pick an increasing sequence $n_1<n_2<\cdots$ such that each $x_{n_t}\in W$ and $x_{n_{t+1}}-x_{n_t}<x_{n_{t+1}}-x_{n_{t+1}}$; this way we ensure that thet set $B=\{x_{n_{t+1}}-x_{n_t} \mid t<\omega\}$ is infinite, and a subset of $D(W)$. Letting $k=\mu(x_{n_0}-x_0)$, we now claim that, for all $t$, we have $\mu(x_{n_{t+1}}-x_{n_t})<k$. To see this, note that $\mu(x_{n_{t+1}}-x_{n_t})<\mu(x_{n_t}-x_0)$, by the definition of $B_1$; from there one can prove inductively that $\mu(x_{n_t}-x_0)=k$ for all $t$ and conclude the claim. But this is a contradiction since $B$ is infinite but $\mu$ is a finite-to-one function. Hence we conclude that $B_0\in D(u)$. Now, if $a\in B_0$, then there are $n<m$ such that $a=x_m-x_n$ and $\mu(x_m-x_n)>\mu(x_n-x_0)$. This implies that $\mu(a)=\mu(x_m-x_n)=\mu(x_m-x_n+x_n-x_0)=\mu(x_m-x_0)$. Thus, by the definition of $A_0$ we must have $f(\mu(a))=f(\mu(x_m-x_0))<\lambda(x_m-x_0)=k_0$, hence $f$ is bounded (by $k_0$) in $\mu[B_0]$ and so, partitioning $\mu[B_0]=\bigcup_{s<k_0}(\mu[B_0]\cap f^{-1}[\{s\}])$, we conclude that $f^{-1}[\{s\}]\in \mu(D(u))$ for some $s<k_0$, that is, the function $f$ is constant on some element of $\mu(D(u))$.
\end{description}
The treatment of the two (exhaustive) cases finishes the proof that $\mu(D(u))$ is a P-point.
\end{proof}

\begin{corollary}
    It is relatively consistent with $\ZFC$ that there are no $\mathbb Z$-Ramsey ultrafilters.\hfill $\square$
\end{corollary}

\section{Some \texorpdfstring{$\mathbb{Z}$}{Lg}-Ramsey ultrafilters are Ramsey}




In this section we prove two results stating certain conditions under which some $\mathbb Z$-Ramsey ultrafilters must necessarily be selective. Whether this implication holds in general remains an open problem.

We first improve on Petrenko and Protasov's result~\cite{MR3681104} that every $(\mathbb Z,4)$-Ramsey ultrafilter is selective.

\begin{theorem}\label{thm:z3impliesramsey}
    If an ultrafilter $u\in\beta\mathbb Z$ is $(\mathbb Z,3)$-Ramsey, then it is selective.
\end{theorem}

\begin{proof}
    Define the colouring $c:[\mathbb Z]^3\longrightarrow 2$ given by $c(\{x,y,z\})=1$ iff $z-y>y-x$, whenever $x<y<z$. Let $A\in u$ be a set such that $[A]^3$ is $c$-monochromatic. Note that, by picking arbitrary $x,y\in A$ with $x<y$, (since $A$ is unbounded) we can always find a $z\in A$ such that $z-y>y-x$, which means that $c(\{x,y,z\})=1$ and so $[A]^3$ is monochromatic in colour $1$.

    Suppose we enumerate increasingly the elements of $A$ as $\{x_n\mid n<\omega\}$. The fact that $[A]^3$ is monochromatic in colour $1$ means that $x_2-x_1>x_1-x_0$ and, in general, $x_{n+1}
    -x_n>x_n-x_1$. That is, the distance between each element of $A$ and the next one is strictly greater than any other distance between two lesser elements. This is easily seen to imply that $A$ must be a thin set, and so the result follows by Theorem~\ref{Zramsey+thinset-implica-ramsey}
\end{proof}

Having proved Theorem~\ref{thm:z3impliesramsey}, it is no longer necessary to consider the notion of $(\mathbb Z,n)$-Ramsey ultrafilters for $n\geq 3$ (as this is already equivalent to the notion of a selective ultrafilter). Therefore, to simplify notation, we will from now on simply write $\mathbb Z$-Ramsey instead of $(\mathbb Z,2)$-Ramsey.

The next theorem to prove, which is the main result of this section, is that every $\mathbb Z$-Ramsey ultrafilter that is also a P-point must be a selective ultrafilter. The remainder of the section is devoted to prove this result.

\begin{lemma}\label{Zramsey-no-arithmetic-progression}
    Let $u$ be a $\mathbb Z$-Ramsey ultrafilter. Then, there exists an $A\in u$ without arithmetic progressions of length three.
\end{lemma}

\begin{proof}
    By Lemma~\ref{lem:mulambdainjective}, there is an $A\in u$ such that, whenever $a,b,c\in A$ with $a<b<c$, we have $\mu(c-b)\neq\mu(b-a)$. In particular, $c-b\neq b-a$, so the elements $a,b,c$ do not form an arithmetic progression.
\end{proof}

Recall that, for subsets $A,B\subseteq\mathbb N$, we denote 
$D(A,B)=\{y-x\mid x\in A,y\in B,\text{ and }x<y\}$. Similarly, for an $x_0\in\mathbb N$ we use the notation $D(x_0,A)=\{y-x_0\mid y\in A\text{ and }x_0<y\}$. Note, e.g., that if we have a set $A$ containing an $x_0\in A$ with $D(x_0,A)\cap D(B,B)\neq\varnothing$, where $B=A\setminus(x_0+1)$, then $A$ is not a thin set.

\begin{lemma}\label{lemma:make-distances-injective}
    Let $u$ be a $\mathbb Z$-Ramsey ultrafilter, let $x_0\in\mathbb N$ and let $A\in u$. Then, there exists a set $B\in u$, with $B\subseteq A$, such that for every two distinct $y,z\in B$, it must be the case that $x_0+z-y\notin B$ (in other words, $D(B,B)\cap D(x_0,B)=\varnothing$).
\end{lemma}

\begin{proof}
    First, use Lemma~\ref{Zramsey-no-arithmetic-progression} to obtain an $A'\subseteq A$, $A'\in u$ containing no nontrivial arithmetic progressions. Define now a colouring $c:[A']^2\longrightarrow 2$ by the formula
    \begin{equation*}
        c(\{x,y\})=\begin{cases}
        1;\text{ if }x_0+y-x\in A'\text{ (that is, if }y-x\in D(x_0,A')\text{)} \\
        0; \text{ otherwise.}
        \end{cases}
    \end{equation*}
    Note that $c$ is $\mathbb Z$-invariant and, therefore, there exists a $B\subseteq A'$, with $B\in u$, such that $[B]^2$ is $c$-monochromatic. We claim that $B$ is as required in the statement of the lemma. This is clearly the case, by the definition of $c$, in case $[B]^2$ is monochromatic in colour $0$ (since $D(x_0,B)\subseteq D(x_0,A')$). So we assume without loss of generality that $[B]^2$ is monochromatic in colour $1$; since $u$ is nonprincipal, assume also that $x_0<\min(B)$. Aiming for a contradiction, suppose that $D(B,B)\cap D(x_0,B)\neq\varnothing$, and pick $x,y,z\in B$ such that $z-y=x-x_0$. Since $x_0$ is the smallest of the four numbers under consideration, necessarily $z$ must be the largest. Note also that we may assume without loss of generality that $x<y<z$ (otherwise, if $x_0<y<x<z$, then we also have that $z-x=y-x_0$ and so we might simply relabel $x$ and $y$). In particular, $x,y,z\in A'$; since $x,y\in B$, then $c(\{x,y\})=1$ and so we must have $w=x_0+y-x\in A'$. But then $y-w=x-x_0=z-y$, which implies that $\{w,y,z\}$ is a nontrivial arithmetic progression in $A'$, a contradiction.
\end{proof}

\begin{corollary}\label{more-general-make-distances-injective}
Let $u$ be a $\mathbb Z$-Ramsey ultrafilter, let $F\in[\mathbb N]^{<\omega}$, and let $A\in u$. Then, there is a $B\in u$, $B\subseteq A$, such that there are no distinct $y,z\in B$ and $w\in B$, $x\in F$ such that $z-y=w-x$ (in other words, such that $D(B,B)\cap D(F,B)=\varnothing$).
\end{corollary}

\begin{proof}
    Induction on $|F|$, the case $|F|=1$ being Lemma~\ref{lemma:make-distances-injective} and, for $|F|>1$, it suffices to pick $x_0\in F$, use $A'$ as in the conclusion of the Corollary for $F\setminus\{x_0\}$, and use again Lemma~\ref{lemma:make-distances-injective} to get a further $B\subseteq A'$, $B\in u$, satisfying the desired property.
\end{proof}

\begin{lemma}\label{makedistancesbig}
Let $u$ be a $\mathbb Z$-Ramsey ultrafilter, and let $d\in\mathbb N$ be arbitrary. Then, there exists an $A\in u$ such that any distance between two elements of $A$ is greater than $d$ (i.e., for any distinct $x,y\in A$ with $x<y$, we have $y-x>d$).
\end{lemma}

\begin{proof}
    Consider the $\mathbb Z$-invariant colouring $c:[\mathbb Z]^2\longrightarrow 2$ given by $c(\{x,y\})=1$ iff $|y-x|>d$. Any $A\in u$ such that $[A]^2$ is $c$-monochromatic will be monochromatic in colour $1$ (since $A$ is infinite and hence unbounded in $\mathbb N$). Therefore the set $A$ is as required.
\end{proof}

Recall the P-point game for a nonprincipal ultrafilter $u$. Players I and II alternate plays; in the $n$-th inning, player I chooses an $A_n\in u$ and then player II chooses a finite $F_n\subseteq A_n$; in the end, player II wins the game if and only if $\bigcup_{n<\omega}F_n\in u$. It is a classical folklore result (attributed to Galvin and McKenzie by Shelah in Chapter VI of \cite{MR1623206}) that a nonprincipal ultrafilter $u$ is a P-point if and only if the P-point game for $u$ is not determined (equivalently, $u$ is a P-point if and only if I does not have a winning strategy for this game, since it is easily shown that player II never has a winning strategy).

\begin{theorem}\label{z-ramsey+ppointimpliesselective}
    Let $u$ be a $\mathbb Z$-Ramsey ultrafilter that is also a P-point. Then, $u$ is a selective ultrafilter.
\end{theorem}

\begin{proof}
    We describe a strategy for player I in the P-point game for $u$.  Begin by playing $A_0=\mathbb N$ and, in the $n+1$-th inning, knowing player II's previous moves $F_0,\ldots,F_n$, satisfying $\max(F_i)<\min(F_{i+1})$ for $i<n$, we play an $A_n\in u$ such that $\min(A_n)>\max(F_n)$, $x,y\in A_n$ and $x<y$ implies $y-x>2\max(F_n)$ (which can be done by Lemma~\ref{makedistancesbig}) and also such that $D(A_{n+1},A_{n+1})\cap D(F_n,A_{n+1})=\varnothing$ (which can be done by Corollary~\ref{more-general-make-distances-injective}). Since $u$ is a P-point, the strategy we just described is not a winning strategy, so there is a run of the game, $\langle F_n\mid n<\omega\rangle$, such that $A=\bigcup_{n<\omega}F_n\in u$.

    Define $c:[A]^2\longrightarrow 2$ given by:
    \begin{equation*}
        c(\{x,y\})=\begin{cases}
            0,\text{ if }x,y\in F_n\text{ for some }n<\omega;\\
            1,\text{ if }x\in F_n\text{ and }y\in F_m\text{ for }m\neq n.
        \end{cases}
    \end{equation*}
    Note that $c$ is a $\mathbb Z$-invariant colouring: for if $x,y,z,w\in A$ are such that $y-x=w-z$ and, say, $x$ is the least of the four numbers, then if $x\in F_n$ we have two cases:
    \begin{description}
        \item[If $y\in F_n$] then it is not possible to have $w,z\in A\setminus F_n\subseteq A_{n+1}$ by construction, so $z\in F_n$ but then it is impossible that $w\in A\setminus F_n\subseteq A_{n+1}$, so it must be the case that $x,y,z,w\in F_n$ and $c(\{x,y\})=0=c(\{z,w\})$.
        \item[If $y\in F_m$ for $m\neq n$] (in which case, we must have $m>n$). Then it is impossible that $w,z\in F_n$, or that $z,w\in A\setminus F_n\subseteq A_{n+1}$, so basically the only option is $z\in F_n$, $w\in A\setminus F_n$ and therefore $c(\{x,y\})=1=c(\{z,w\})$.
    \end{description}
Note, also, that any set $X$ that is $c$-monochromatic is either monochromatic in colour $0$ and finite (contained within one single $F_n$), or infinite and thus monochromatic in colour $1$, in which case $X$ is a selector for the family $\{F_n\mid n<\omega\}$. Therefore, if we pick an $X\in u$, with $X\subseteq A$ and such that $[X]^2$ is $c$-monochromatic, then $X$ must be a selector for the family $\{F_n\mid n<\omega\}$. Thus, by construction, if $x,y,z,w\in X$ with $x<y$, $z<w$, and $x<z$, then if $x\in F_n$, we must have that $x,z,w\in A\setminus F_n\subseteq A_{n+1}$, which makes impossible to have $y-x=w-z$. Therefore, $X$ is a thin set, which implies, by Theorem~\ref{Zramsey+thinset-implica-ramsey}, that $u$ is a selective ultrafilter.
\end{proof}

\section{A partition result for translation-invariant colourings of the integers}

In this section we will complement the results from Section 2 and end with a question.

We begin by recalling the version of the Central Sets Theorem that we will use in this section (this particular formulation can be found in~\cite{hindman-de-strauss}, for a proof see~\cite[Prop. 8.21]{furstenberg}).

\begin{theorem}\label{thm:centralsets}
    Let $A$ be a central set. Then, whenever we have finitely many sequences $\langle a_n^1\mid n<\omega\rangle,\cdots,\langle a_n^k\mid n<\omega\rangle$, there exists a sequence $\langle x_n\mid n<\omega\rangle$ and a block sequence $\langle H_n\mid n<\omega\rangle$ of finite subsets of $\omega$ such that, for every $t\in\{1,\ldots,k\}$ we have
    \begin{equation*}
        \fs\left(x_n+\sum_{i\in H_t}a_t^i\mid n<\omega\right)\subseteq A.
    \end{equation*}
\end{theorem}

In fact, sets satisfying the conclusion of the previous Theorem are called J-sets~\cite[Def.14.8.1]{hindman-strauss}. A key result relating these combinatorially rich sets is that every piecewise syndetic set (in particular, every central set) is a J-set~\cite[Theorem 14.8.3]{hindman-strauss}. Here we focus on central sets, defined as sets that belong to a minimal idempotent element of $\beta\omega$, since they will be the ones that we utilize in our proof.

The main theorem of the section is the following.

\begin{theorem}\label{thm:positivemonochromatic}
    For every $\mathbb Z$-invariant colouring $c:[\mathbb Z]^2\longrightarrow r$ there exists a $c$-monochromatic set $X$ such that $X\notin\text{$k$-$\mathsf{thin}$}$ for all $k$.
\end{theorem}

The remainder of the section will be used to prove this theorem. For this we introduce the following definitions.

\begin{definition}
Let $s=\langle d,a_1,\ldots,a_l\rangle$ be a sequence of natural numbers of length $l+1$.
\begin{enumerate}
    \item Given a triple $(i,k,t)$ of elements of $\omega$, we define
\begin{equation*}
    (i,k,t)*s=a_{i+1}+a_{i+2}+\cdots+a_{i+k}+td,
\end{equation*}
with the convention that, if $k=0$, then $(i,k,t)*\vec{a}=td$.
    \item The {\bf $l$-pattern generated by $s$} is the set
\begin{equation*}
    L(s)=\left\{(i,k,t)*\vec{a}\mid 0\leq k\leq l,0\leq i<l,\text{ and }k-1\leq t\leq k+1\right\}.
\end{equation*}
\end{enumerate}
\end{definition}

The motivation for these definitions comes from an analysis of $k$-thin sets and difference sets. More concretely, suppose we have a set $X$ whose elements, enumerated increasingly, are $x_0<x_1<\cdots<x_{2l}<x_{2l+1}$; suppose furthermore that there is a sequence $s=(d,a_1,\ldots,a_l)$ such that $x_{2k+1}-x_{2k}=d$ and $x_{2k+2}-x_{2k+1}=a_k$ for all $k$. The fact that so many distances between consecutive elements of $X$ equal $d$ shows that $X$ is not an $l$-thin set. Furthermore, it is not hard to check that in this case we will have $D(X)=L(s)$. So $l$-patterns arise naturally from considerations regarding how to build sets that are not $l$-thin.

\begin{proposition}
Let $s=(d,a_1,\ldots,a_l,a_{l+1})$ be an $(l+2)$-sequence, and denote by $t=(d,a_1,\ldots,a_l)$ the $(l+1)$-sequence that results from dropping the last entry of $s$. Then, the $(l+1)$-pattern generated by $s$ is given by
\begin{equation*}
    L(s)=L(t)\cup\{y+a_{l+1}\mid y\in L(t)\}\cup\{y+a_{l+1}+d\mid y\in L(t)\}\cup\{a_{l+1},a_{l+1}+d\}.
\end{equation*}
\end{proposition}

\begin{proof}
Straightforward.
\end{proof}

The following theorem ensures that central sets contain $l$-patterns for arbitrarily large values of $l$. The proof is done by induction on $l$, and the inductive hypothesis that we need is a little stronger, so this is reflected in the way we state the theorem.

\begin{theorem}\label{thm:arbitrarilylong}
Let $A$ be a central set, and let $l<\omega$. Then there exist sequences $\langle d_n\mid n<\omega\rangle,\langle a_n^1\mid n<\omega\rangle,\ldots,\langle a_n^l\mid n<\omega\rangle$ such that, for all $0\leq i<l,0\leq k\leq l,k-1\leq t\leq k+1$, letting $s_l^n=\langle d_n,a_n^1,\ldots,a_n^l\rangle$ and $\vec{x}^{(i,k,t)}=\langle (i,k,t)*s_l^n\mid n<\omega\rangle$, we have $\fs(\vec{x}^{(i,k,t)})\subseteq A$. In particular (since each element of $L(s^n)$ is the $n$-th term of one of the sequences $\vec{x}^{(i,k,t)}$), the set $A$ contains $l$-patterns for arbitrarily large $l$.
\end{theorem}

\begin{proof}
The proof goes by induction on $l$. For $l=0$ we only need to take a sequence $\langle d_n\mid n<\omega\rangle$ such that $\fs(\langle d_n\mid n<\omega\rangle)\subseteq A$, which is simply Hindman's theorem (since $A$ is central, in particular it is an IP-set).

Now suppose the theorem already holds for some $l$ and let $\langle \delta_n\mid n<\omega\rangle,\langle \alpha_n^1\mid n<\omega\rangle,\ldots,\langle \alpha_n^l\mid n<\omega\rangle$ be the sequences as in the statement of the theorem. We now apply the Central Sets Theorem~\ref{thm:centralsets} to the set $A$ along with the sequences
\begin{equation*}
    \langle(i,k,t)*s_l^n+\delta_n\mid n<\omega\rangle
\end{equation*}
and
\begin{equation*}
    \langle(i,k,t)*s_l^n+2\delta_n\mid n<\omega\rangle
\end{equation*}
for each $(i,k,t)$ such that $1\leq i<l$, $0\leq k\leq l$, and $k-1\leq t\leq k+1$; this yields a sequence $\vec{x}=\langle x_n\mid n<\omega\rangle$ of numbers and a block sequence $\langle H_n\mid n<\omega\rangle$ of finite sets satisfying the statement of Theorem~\ref{thm:centralsets}. Define, for each $n<\omega$,
\begin{eqnarray*}
d_n & = & \sum_{i\in H_n}\delta_i, \\
a_n^j & = & \sum_{i\in H_n}\alpha_i^j, \ \ \ \text{ for }1\leq j\leq l, \\
a_n^{l+1} & = & x_n+\sum_{i\in H_n}\delta_i.
\end{eqnarray*}

We now show that, with $\vec{x}^{(i,k,t)}$ defined as in the statement of the theorem for all of the $l+2$ recently obtained sequences, we have $\fs(\vec{x}^{(i,k,t)})\subseteq A$. If $i+k\leq l$ then the desired result follows immediately by the inductive hypothesis. On the other hand, if $i+k=l+1$, then by looking at the terms of $\vec{x}^{(i,k,t)}$ we obtain 
\begin{equation*}
\vec{x}^{(i,k,t)}=\begin{cases}
\langle x_n+\sum_{i\in H_n}\delta_i\rangle,\ \ \ \text{ if }(i,k,t)=(l,1,0), \\
\langle x_n+\sum_{i\in H_n}2\delta_i\rangle,\ \ \ \text{ if }(i,k,t)=(l,1,1), \\
\langle x_n+\sum_{i\in H_n}\left((i,k-1,t)*s_l^n+\delta_i\right)\rangle,\ \ \ \text{ if }i\leq l\text{ and }t\leq k, \\
\langle x_n+\sum_{i\in H_n}\left((i,k-1,t)*s_l^n+2\delta_i\right)\rangle,\ \ \ \text{ if }i\leq l\text{ and }t=k+1.
\end{cases}
\end{equation*}
All of these sequences have their set of finite sums contained within $A$, by the choice of the sequence of $x_n$ utilizing the Central Sets Theorem.
\end{proof}

With the previous theorem in hand, we are in good shape for finally providing a proof of Theorem~\ref{thm:positivemonochromatic}.

\begin{proof}[Proof of Theorem~\ref{thm:positivemonochromatic}]
Given a $\mathbb Z$-invariant colouring $c:[\mathbb Z]^2\longrightarrow 2$, define a colouring of the natural numbers $\tilde{c}:\mathbb N\longrightarrow 2$ by $\tilde{c}(n)=c(0,n)$ (so that also $\tilde{c}(n)=c(k,n+k)$ for every $k$ due to the $\mathbb Z$-invariance of $c$). Let $u$ be a minimal idempotent in $\beta\mathbb N$, and let $A\in u$ be $\tilde{c}$-monochromatic. Then $A$ is a central set. We will find a set $X\subseteq\mathbb Z$ such that $D(X)$ contains arbitrarily long $l$-patterns, and such that $D(X)\subseteq A$; this will imply on the one hand that $X$ is not $k$-thin for any $k$, and on the other hand that $c[[X]^2]=\tilde{c}[D(X)]\subseteq A$ (and hence $X$ will be $c$-monochromatic). To do this, we recursively define an increasing sequence of finite sets $X_l$, containing some $l$-pattern, such that $D(X_l)\subseteq A^*=\{n\in A\mid A-n\in u\}$. In the first step of the induction, it suffices to let $X_0$ be any set with two elements whose distance belongs to $A^*$. Now assume we already have our finite set $X_{l-1}$ with $D(X_{l-1})\subseteq A^*$. Then the set
\begin{equation*}
    B=A^*\cap\left(\bigcap_{y\in D(X_{l-1})}(A-y)\right)
\end{equation*}
belongs to $u$, and is therefore central. By Theorem~\ref{thm:arbitrarilylong} there is a sequence $s=(d,a_1,\ldots,a_l)$ such that $L(s)\subseteq B$. Arbitrarily pick numbers $\max(X_{l-1})<x_0<x_1<\cdots<x_2l<x_{2l+1}$ such that $x_{2k+1}-x_{2k}=d$ and $x_{2k+2}-x_{2k+1}=a_k$ for all $k$, and let $X_l=X_{l-1}\cup\{x_0,\ldots,x_{2l+1}\}$. It is readily checked that each element of$ D(X_l)$ is either an element of $L(s)$ (any distance between two of the $x_k$), or an element of $D(X_{l-1})$ (any distance between two elements of $X_{l-1}$), or an element of the form $y+x$ with $y\in D(X_{l-1})$ and $x\in L(s)$ (any distance between an element of $X_{l-1}$ and one of the $x_k$). In either case (the first case by the choice of $B$ and $s$, the second case by induction hypothesis, and the third case because $L(s)\subseteq A-y$ for any $y\in D(X_{l-1})$ by construction) we can conclude that $D(X_l)\subseteq A^*$, and the induction can continue. In the end, it suffices to make $X=\bigcup_{n<\omega}X_l$.
\end{proof}

\begin{question}
    Is it the case that for every $\mathbb Z$-invariant colouring $c:[\mathbb Z]^2\longrightarrow 2$ there exists a $c$-monochromatic set $X\in \mathcal{I}_{\text{$k$-$\mathsf{thin}$}}^+$, that is,
    \[
    \mathbb{Z} \overset{\text{left-inv.}}{\xrightarrow{\hspace{1.1cm}}} (\mathcal{I}_{\text{$k$-$\mathsf{thin}$}}^+)^{2}_{2} \,?
    \]
\end{question}



\bibliographystyle{amsplain}

\bibliography{References}

\end{document}